\documentclass[reqno]{amsart}

\usepackage[letterpaper,top=2cm,bottom=2cm,left=3cm,right=3cm]{geometry}
\usepackage[T1]{fontenc}
\usepackage[utf8]{inputenc}
\usepackage[english]{babel}
\usepackage{lmodern}
\usepackage{microtype}
\usepackage{amsmath,amssymb,amsthm,mathrsfs,mathtools}
\usepackage{booktabs}
\usepackage{graphicx}
\usepackage{xcolor}
\usepackage[colorlinks=true,allcolors=blue]{hyperref}

\theoremstyle{plain}
\newtheorem{theorem}{Theorem}[section]
\newtheorem{proposition}[theorem]{Proposition}
\newtheorem{lemma}[theorem]{Lemma}
\newtheorem{corollary}[theorem]{Corollary}

\theoremstyle{definition}

\newtheorem{remark}[theorem]{Remark}
\newtheorem{problem}[theorem]{Problem}

\newcommand{\defin}[1]{%
\relax\ifmmode%
\textcolor{blue}{#1}%
\else\textcolor{blue}{\emph{#1}}%
\fi%
}

\newcommand{\wvec}{\mathbf{w}}
\newcommand{\PF}{\operatorname{PF}}
\newcommand{\Smir}{\operatorname{Smir}}
\newcommand{\NC}{\operatorname{NC}}
\newcommand{\asc}{\operatorname{asc}}
\newcommand{\wasc}{\operatorname{wasc}}
\newcommand{\des}{\operatorname{des}}
\newcommand{\tie}{\operatorname{tie}}
\newcommand{\pk}{\operatorname{pk}}
\newcommand{\Des}{\operatorname{Des}}
\newcommand{\Tie}{\operatorname{Tie}}
\newcommand{\Cat}{\operatorname{Cat}}
\newcommand{\JacP}{\mathsf{P}}
\newcommand{\lless}{\mathrel{\ll}}
\newcommand{\oeis}[1]{\href{https://oeis.org/#1}{#1}}

\title[Real-rootedness and interlacing]
{Real-rootedness and interlacing for parking functions and Chow polynomials}

\author{Per Alexandersson}
\address{Department of Mathematics, Stockholm University,
  SE-106 91 Stockholm, Sweden}
\email{per.w.alexandersson@gmail.com}


\begin{document}

\begin{abstract}
We prove real-rootedness for the Chow polynomials of the noncrossing partition
lattices by transferring tieless parking functions to finite-alphabet Smirnov
words and applying an interlacing-preserving transition of M.~Leander.
We also derive a
triangular recurrence for peaks and ties and identify the peakless-tieless
descent polynomial as the Narayana polynomial.  For the toric
$g$-contributions of Ehrenborg--Hetyei--Readdy, we exhibit a fixed-row common
interlacer and establish real-rootedness of all nonnegative row sums.
Individual real-rootedness follows in particular; Q.~Xiao also proved it by a
different differential recurrence.  We also give a
second proof of the individual statement, by finite Schur--Szegő convolution,
that does not use the common interlacer.
These results prove Conjecture~4.2 of Xiao and
Conjecture~11.2 of Ehrenborg--Hetyei--Readdy, with consequences for weakly
123-avoiding parking functions.  We also prove real-rootedness for the
image-size polynomial on all parking functions and for the ascent and descent
polynomials of four two-pattern-avoiding classes.
\end{abstract}

\maketitle

\section{Introduction}\label{sec:introduction}

A \defin{parking function} is a word
$\wvec=(w_1,\dotsc,w_n)\in[n]^n$ whose nondecreasing rearrangement
$b_1\leq\dotsb\leq b_n$ satisfies $b_i\leq i$ for every $i$.
We write $\PF_n$ for the set of parking functions of length $n$.  This is a
classical family of cardinality $(n+1)^{n-1}$; we refer to
C.~H.~Yan~\cite{Yan2015} for a survey.

Several recent questions concern the local comparison statistics of parking
functions.  P.~R.~F.~Schumacher~\cite{Schumacher2018} studied their descents
and ties.  A.~Cruz, P.~E.~Harris, K.~J.~Harry, J.~Kretschmann,
M.~McClinton, A.~Moon, J.~O.~Museus, and
E.~Redmon~\cite{CruzHarrisHarryEtAl2024} subsequently studied descent sets, peaks,
valleys, and statistic encodings.  In particular, they asked for recursive
or closed formulas for the joint distribution of peaks and ties.

The descent polynomial of tieless parking functions has also appeared in a
different setting.  L.~Ferroni, J.~P.~Matherne, and L.~Vecchi proved that the
Chow polynomial of every finite graded bounded poset is palindromic and has
nonnegative, unimodal coefficients.  They also gave examples that are not
real-rooted and conjectured real-rootedness for every Cohen--Macaulay
poset~\cite{FerroniMatherneVecchi2024}.  Real-rootedness is known for several
classes, including uniform geometric lattices and maximal ranked
posets~\cite{BrandenVecchi2026} and UMEL-shellable
posets~\cite{CoronFerroniLi2025}, as well as totally nonnegative posets and
lattices of flats of paving matroids
\cite[Theorem~6.1 and Corollary~7.2]{BrandenVecchi2025TN}.  The ordinary
partition lattice belongs to the UMEL-shellable class
\cite[Theorem~1.8]{CoronFerroniLi2025}, whereas
$\NC_{n+1}$ does not when $n\geq3$~\cite[Section~1]{ChunHayashidaPartidaWang2026};
thus these poset-class results do not apply directly here.  A separate
Toeplitz-matrix theorem of P.~Br\"and\'en and L.~Vecchi does imply
real-rootedness of every finite-alphabet Smirnov descent polynomial
\cite[Theorems~8.4 and~8.11]{BrandenVecchi2025TN}.

J.~Chun, N.~Hayashida, E.~Partida, and Z.~Wang proved that the descent
polynomial of tieless parking functions equals the Chow polynomial of the
noncrossing partition lattice
$\NC_{n+1}$~\cite[Theorem~1.1]{ChunHayashidaPartidaWang2026}.  They proved
$\gamma$-positivity and conjectured real-rootedness
in~\cite[Conjecture~1.3]{ChunHayashidaPartidaWang2026}.

The paper has three main themes.  First, we establish a transfer from tieless
parking functions to finite-alphabet Smirnov words and combine it with
Leander's interlacing-preserving transition to prove real-rootedness for the
noncrossing Chow polynomials.
Second, we study peaks and ties through a finite insertion recurrence and a
Narayana refinement.  Third, we prove
fixed-row interlacing and common-interlacer results for toric
$g$-contributions, with consequences for nonnegative $\gamma$-weighted sums
and weakly 123-avoiding parking functions.  We end with two shorter
applications to image size and two-pattern avoidance.

The first theme settles the conjecture of Chun, Hayashida, Partida, and Wang.

\begin{theorem}\label{thm:mainChow}
For every $n\geq1$, the Chow polynomial $H_{\NC_{n+1}}(t)$ is real-rooted.
\end{theorem}

The proof has two steps.  We first group contents by multiplicity type and use
the symmetry of the Smirnov word descent enumerator to prove
\begin{equation}\label{eq:introSmirnovTransfer}
H_{\NC_{n+1}}(t)
=\frac{1}{n+1}
\sum_{\substack{\wvec\in[n+1]^n\\
                  w_i\neq w_{i+1}\ (1\leq i<n)}}
t^{\des(\wvec)}.
\end{equation}
After reversing the alphabet order, the final-letter recurrence is the
transition in~\cite[Lemma~2.1]{Leander2016}, now started from an all-one vector
because the first letter is unrestricted.  Leander's preservation theorem
therefore gives a last-letter interlacing vector.  Combined with the transfer
identity, this proves the theorem.  The scalar Smirnov real-rootedness needed
here also follows from the more general Toeplitz-matrix results of Br\"and\'en
and Vecchi.  We give the details in Section~\ref{sec:interlacing}.

The second theme gives a recursive answer to Problems~3 and~4 of Cruz et
al.~\cite{CruzHarrisHarryEtAl2024}.  We bound the maximal letter and insert all
copies of the new maximum.  The resulting weights depend only on the numbers
of peaks and ties.

For $1\leq m\leq n$, let
\[
\defin{H_{n,m}(t,u)}
\coloneqq
\sum_{\substack{\wvec\in\PF_n\\\max_i w_i\leq m}}
t^{\pk(\wvec)}u^{\tie(\wvec)}.
\]
\begin{theorem}\label{thm:parkingPeakTieRecurrence}
For $2\leq m\leq n$,
\begin{equation}\label{eq:parkingPeakTieRecurrence}
H_{n,m}(t,u)
=H_{n,m-1}(t,u)
+\sum_{r=1}^{n-m+1}
\mathcal I_{n-r,r}H_{n-r,m-1}(t,u),
\end{equation}
with initial condition
\begin{equation}\label{eq:parkingPeakTieInitial}
H_{n,1}(t,u)=u^{n-1},
\end{equation}
where the explicit insertion operator $\mathcal I_{\ell,r}$ is given
in~\eqref{eq:insertionOperator}.
\end{theorem}

We also refine the Catalan enumeration of peakless-tieless parking
functions in~\cite{CruzHarrisHarryEtAl2024}.  Their bijection sends ties of a
nondecreasing parking function to descents, which gives the Narayana
polynomial.

The first two themes share a common principle: local comparison statistics
become tractable after choosing a finite refinement adapted to the statistic.
For tieless descents, we first group contents by multiplicity type and then
refine Smirnov words by their final letter.  For peaks and ties, we bound the
maximum and distinguish the local types of insertion gaps.

The third theme concerns the universal toric $g$-contributions of Ehrenborg,
Hetyei, and Readdy.  They conjectured that every contribution polynomial is
real-rooted
\cite[Conjecture~11.1]{EhrenborgHetyeiReaddy2025}.  Q.~Xiao proved this
conjecture, together with two adjacent-rank interlacings, by a
differential recurrence derived from the Catalan generating function and a
Liu--Wang-type induction~\cite[Theorem~1.3]{Xiao2026}.  Beyond individual
real-rootedness, Theorem~\ref{thm:toricContributionInterlacing} constructs the
explicit common interlacer $g_{n,\lfloor n/2\rfloor}(t)/t$ and proves
that each fixed row is an interlacing sequence, as well as real-rootedness of
every nonnegative row sum; individual real-rootedness is an immediate
corollary.
We retain a separate finite-convolution proof of that corollary in
Appendix~\ref{app:independentToricRealRootedness}.
The fixed-row assertion, together with the two trivial initial rows, resolves
the conjecture formulated by Xiao at the suggestion of C.~Athanasiadis
\cite[Conjecture~4.2]{Xiao2026}.  As Xiao observed
\cite[Section~4]{Xiao2026}, this assertion implies real-rootedness for the
toric $g$-polynomial of every simple polytope with nonnegative
$\gamma$-vector.  In particular, it proves Conjecture~11.2 of
Ehrenborg--Hetyei--Readdy and, through their associahedral identity, gives
real-rootedness of the weak-ascent and strict-descent polynomials on weakly
123-avoiding parking functions.  The weak-ascent coefficient array is
OEIS~\oeis{A390883}~\cite{OEIS2026}.  The common-interlacer argument has also
been checked in Lean~\cite{AlexanderssonRealRooted}.

Two further consequences are independent of the main arguments.  Pollak's
cyclic action and an occupancy recurrence give
simple negative zeros for the image-size polynomial on all parking functions.
A labeled-Dyck-path interpretation of pattern avoidance gives the rows of
OEIS~\oeis{A033282} and their reciprocals, the rows of
OEIS~\oeis{A126216}.

The paper is outlined as follows.  Section~\ref{sec:statistics} fixes the
parking-function statistics.  The first group of results occupies
Sections~\ref{sec:transfer} and~\ref{sec:interlacing}: the former proves the
Smirnov transfer, and the latter establishes the last-letter recurrence and
proves Theorem~\ref{thm:mainChow}.  The second group occupies
Sections~\ref{sec:peakTie} and~\ref{sec:narayana}, which prove the
maximal-letter insertion recurrence and the Narayana refinement, respectively.
Section~\ref{sec:ordinaryDescents} records an additional observation about
ordinary descents.  The third group is Section~\ref{sec:toricContributions},
which constructs a common interlacer by a differential insertion triangle and
a boundary signed-moment argument, then proves fixed-row interlacing.
Appendix~\ref{app:independentToricRealRootedness} gives the separate
finite-convolution proof of individual real-rootedness.
Section~\ref{sec:furtherRealRootedness} records the image-size and
two-pattern consequences.  Finally, Section~\ref{sec:questions} records
further questions.

\section{Parking-function statistics}\label{sec:statistics}

For $\wvec\in\PF_n$ and $i\in[n-1]$, we say that $i$ is an
\defin{ascent}, \defin{descent}, or \defin{tie} according as
\[
w_i<w_{i+1},\qquad w_i>w_{i+1},\qquad\text{or}\qquad w_i=w_{i+1}.
\]
We write $\Des(\wvec)$ and $\Tie(\wvec)$ for the descent and tie sets, and
$\asc(\wvec)$, $\des(\wvec)$, and $\tie(\wvec)$ for the respective numbers
of ascents, descents, and ties.

\subsection{Peaks}\label{subs:peaks}

For $2\leq i\leq n-1$, we call $i$ a \defin{peak} if
\[
w_{i-1}<w_i>w_{i+1}.
\]
We write $\pk(\wvec)$ for the number of peaks.  Thus we use the strict-peak
convention throughout; on tieless words it agrees with the weak-left
convention $w_{i-1}\leq w_i>w_{i+1}$.

\section{Contents and the Smirnov transfer}\label{sec:transfer}

The \defin{content} of a word $\wvec$ is the sequence
$c(\wvec)=(c_1,c_2,\dotsc)$, where $c_j$ is the number of occurrences of $j$.
Its \defin{multiplicity type} is the partition $\lambda\vdash n$ obtained by
sorting the positive entries of $c(\wvec)$.  We write $\ell(\lambda)$ for the
number of parts of $\lambda$ and $m_s(\lambda)$ for the multiplicity of the
part $s$.  Define the \defin{Kreweras number} by
\begin{equation}\label{eq:KrewerasNumber}
\defin{K_\lambda}
\coloneqq
\frac{n!}{(n-\ell(\lambda)+1)!
\prod_{s\geq1}m_s(\lambda)!}.
\end{equation}

A \defin{Smirnov word} is a word with no equal adjacent letters.  Let
$\Smir_{r,m}$ denote the set of Smirnov words of length $r$ on the alphabet
$[m]$.

Let $G_r$ be the naturally labeled path on $[r]$.  Its \defin{chromatic
quasisymmetric function} is
\[
\defin{X_{G_r}(\mathbf x,t)}
\coloneqq
\sum_{\kappa\in\mathcal C(G_r)}
t^{\asc_{G_r}(\kappa)}x_{\kappa(1)}\dotsm x_{\kappa(r)},
\]
where $\mathcal C(G_r)$ is the set of proper positive-integer colorings and
$\asc_{G_r}(\kappa)$ counts the edges $\{i,i+1\}$ with
$\kappa(i)<\kappa(i+1)$.  Reading a coloring in reverse gives a Smirnov word
with the same monomial and turns these ascents into descents.  Hence
\begin{equation}\label{eq:pathSmirnovEnumerator}
X_{G_r}(\mathbf x,t)
=\sum_{\substack{\wvec\text{ Smirnov}\\\text{of length }r}}
t^{\des(\wvec)}x_{w_1}\dotsm x_{w_r}.
\end{equation}
J.~Shareshian and M.~L.~Wachs identify this path function with the Smirnov word
enumerator and give the symmetric-function identity
\begin{equation}\label{eq:pathSmirnovGeneratingSeries}
1+\sum_{r\geq1}X_{G_r}(\mathbf x,t)z^r
=\frac{(1-t)E(z)}{E(tz)-tE(z)},
\qquad
E(z)=\sum_{r\geq0}e_r(\mathbf x)z^r;
\end{equation}
see~\cite[Example~2.5]{ShareshianWachs2016}.  In particular,
$X_{G_r}(\mathbf x,t)$ is symmetric in $\mathbf x$.

\begin{lemma}\label{lem:SmirnovContentSymmetry}
The descent polynomial of Smirnov words with a fixed content depends only on
its multiplicity type.
\end{lemma}

\begin{proof}
For a content $\mathbf c=(c_1,c_2,\dotsc)$, its descent polynomial is the
coefficient of $x_1^{c_1}x_2^{c_2}\dotsm$ in
$X_{G_r}(\mathbf x,t)$.  The claim follows from the symmetry of this function.
\end{proof}

\begin{lemma}\label{lem:typePlacementRatio}
Let $\lambda\vdash n$ have $b$ parts.  The number of contents of type
$\lambda$ on the alphabet $[n+1]$ is $(n+1)K_\lambda$.
\end{lemma}

\begin{proof}
We first choose the $b$ letters in the support, and then assign the parts of
$\lambda$ to these letters.  Repeated parts are indistinguishable, so the
number of contents is
\[
\binom{n+1}{b}\frac{b!}{\prod_{s\geq1}m_s(\lambda)!}
=\frac{(n+1)!}{(n+1-b)!\prod_{s\geq1}m_s(\lambda)!}.
\]
Comparing this expression with~\eqref{eq:KrewerasNumber} gives
$(n+1)K_\lambda$, as desired.
\end{proof}

\begin{lemma}\label{lem:KrewerasContentCount}
The number of nondecreasing parking functions of multiplicity type $\lambda$
is $K_\lambda$.
\end{lemma}

\begin{proof}
Let $c=(c_1,\dotsc,c_n)$ be a content vector of type $\lambda$.  It is a
parking-function content if and only if
\[
\sum_{i=1}^j(c_i-1)\geq0
\qquad (1\leq j\leq n).
\]
Append one more entry $-1$ to the sequence
$(c_1-1,\dotsc,c_n-1)$.  The resulting sequence has total sum $-1$; its
entries $s-1$ occur with multiplicity $m_s(\lambda)$, and $-1$ occurs
$n-\ell(\lambda)+1$ times.  The cycle lemma~\cite{Raney1960} says that
exactly one of the $n+1$ cyclic shifts has all proper partial sums
nonnegative.  If $s$ is the sum of the first $n$ entries of this shift and
$x$ is its last entry, then $s\geq0$, $x\geq-1$, and $s+x=-1$.
Consequently, $s=0$ and $x=-1$.  Deleting the last entry recovers exactly
one parking-function content.  Since a sequence of total sum $-1$ has no
nontrivial period, every cyclic orbit has size $n+1$.  Hence the desired
number is
\[
\frac{1}{n+1}
\frac{(n+1)!}{(n-\ell(\lambda)+1)!
\prod_{s\geq1}m_s(\lambda)!}
=K_\lambda.
\]
This is also Kreweras's fixed-block-type formula for noncrossing partitions;
see~\cite[Theorem~4]{Kreweras1972}.
\end{proof}

Thus the same numbers $K_\lambda$ count parking-function contents of type
$\lambda$, while $(n+1)K_\lambda$ counts placements of that type on the
alphabet $[n+1]$.

\begin{theorem}\label{thm:SmirnovTransfer}
For every $n\geq1$,
\begin{equation}\label{eq:SmirnovTransfer}
\sum_{\substack{\wvec\in\PF_n\\\tie(\wvec)=0}}t^{\des(\wvec)}
=\frac{1}{n+1}\sum_{\wvec\in\Smir_{n,n+1}}t^{\des(\wvec)}.
\end{equation}
\end{theorem}

\begin{proof}
We partition both sides by multiplicity type $\lambda\vdash n$.  A parking
function content of type $\lambda$ contributes the same Smirnov descent
polynomial as every other content of that type, by
Lemma~\ref{lem:SmirnovContentSymmetry}.  There are $K_\lambda$ such
parking-function contents by Lemma~\ref{lem:KrewerasContentCount}.

On the word side, Lemma~\ref{lem:typePlacementRatio} gives
$(n+1)K_\lambda$ contents of type $\lambda$.  Thus the contribution of every
multiplicity type on the right is exactly $n+1$ times its contribution on the
left.  Summing over $\lambda$ proves~\eqref{eq:SmirnovTransfer}.
\end{proof}

\begin{remark}\label{rem:transferLimitation}
The symmetry hypothesis in this proof is essential.  Ordinary multiset
Eulerian polynomials are symmetric in the multiplicities, and
Lemma~\ref{lem:SmirnovContentSymmetry} gives the corresponding fact for
tieless descents.  A peak enumerator of multiset words can depend on
which multiplicity is assigned to which position in the ordered alphabet.
Thus Theorem~\ref{thm:SmirnovTransfer} does not give a peak transfer.
\end{remark}

Combining Theorem~\ref{thm:SmirnovTransfer}
with~\cite[Theorem~1.1]{ChunHayashidaPartidaWang2026}
gives~\eqref{eq:introSmirnovTransfer}.

\section{Smirnov interlacing and noncrossing Chow polynomials}
\label{sec:interlacing}

The transition below already appears in M.~Leander's work on restricted
Smirnov words.  Her words have fixed first letter $0$, whereas our first letter
is unrestricted.  After reversing the alphabet order, her Lemma~2.1 gives the
same transition as~\eqref{eq:lastLetterRecurrence}, with different initial
data.  Concretely, Leander's boundary condition gives the length-one vector
$(0,t,\dotsc,t)$, whereas our unrestricted first letter gives the all-one
vector.  Her Theorem~2.3 proves that the transition preserves the relevant
compatibility conditions, and Corollary~3.5 gives its interlacing-matrix
formulation~\cite{Leander2016}.  Leander applies this machinery to her
fixed-initial-letter refinements in Theorem~2.4.  Starting it instead from the
all-one vector gives the last-letter statement recorded below.

Separately, Br\"and\'en and Vecchi's results imply that the polynomial on the
right of~\eqref{eq:SmirnovTransfer} is real-rooted for every word length and
alphabet size.  Indeed, specialize their supersymmetric variables to
\[
x_1=\dotsb=x_m=1,\qquad x_i=0\ (i>m),\qquad y_i=0\ (i\geq1).
\]
Then their series is $f(z)=(1+z)^m$, and the words surviving in their
Theorem~8.11 are precisely the Smirnov words on $[m]$, with the same descent
weight and no collision factor.  Their Theorem~8.4 therefore gives the
unrefined real-rootedness~\cite{BrandenVecchi2025TN}.

For $1\leq j\leq m$, define the \defin{reversed last-letter refinement}
\begin{equation}\label{eq:lastLetterDefinition}
\defin{F^{(m)}_{r,j}(t)}
\coloneqq
\sum_{\substack{\wvec\in\Smir_{r,m}\\w_r=m+1-j}}t^{\des(\wvec)}.
\end{equation}
\begin{proposition}\label{prop:lastLetterRecurrence}
For $r\geq1$ and $1\leq j\leq m$,
\begin{equation}\label{eq:lastLetterRecurrence}
F^{(m)}_{1,j}(t)=1,
\qquad
F^{(m)}_{r+1,j}(t)
=t\sum_{k<j}F^{(m)}_{r,k}(t)
+\sum_{k>j}F^{(m)}_{r,k}(t).
\end{equation}
\end{proposition}

\begin{proof}
A word of length one has no descent.  For the recurrence, remove the final
letter $m+1-j$.  If the preceding final letter is $m+1-k$ with $k<j$,
appending $m+1-j$ creates a descent and contributes a factor of $t$.  If
$k>j$, it creates an ascent.  The case $k=j$ is excluded by the Smirnov
condition.  These cases give~\eqref{eq:lastLetterRecurrence}.
\end{proof}

For real-rooted polynomials $f$ and $g$ with positive leading coefficients,
list their zeros with multiplicity as
$\alpha_1\geq\dotsb\geq\alpha_r$ and
$\beta_1\geq\dotsb\geq\beta_s$, respectively.  We write
$\defin{f\lless g}$ if $s\in\{r,r+1\}$ and their zeros alternate as
$\beta_{r+1}\leq\alpha_r\leq\beta_r\leq\dotsb
\leq\alpha_1\leq\beta_1$,
where the leftmost term is omitted when $s=r$; for $r=0$, only the degree
condition $s\in\{0,1\}$ remains.
We say that $f$ \defin{interlaces} $g$ in this case.  We follow the orientation
of~\cite[Section~3]{Leander2016}, with the degree condition made explicit;
equalities are allowed in the
displayed root order.  A
degree-$r$ polynomial $h$ is a \defin{common interlacer} of degree-$(r+1)$
polynomials $g_1,\dotsc,g_q$ if $h\lless g_i$ for every $i$.  An
\defin{interlacing sequence} is a sequence $(f_1,\dotsc,f_m)$ such that
$f_i\lless f_j$ whenever $i<j$.

Leander's preservation theorem applies directly to
\eqref{eq:lastLetterRecurrence}, with the all-one vector as initial data
\cite[Corollary~3.5]{Leander2016}.
The general matrix criterion underlying this step is also formalized in the
Lean library~\cite{AlexanderssonRealRooted}.

\begin{theorem}\label{thm:SmirnovRealRooted}
For all $r\geq1$ and $m\geq2$, the vector
\[
\bigl(F^{(m)}_{r,1}(t),F^{(m)}_{r,2}(t),\dotsc,
F^{(m)}_{r,m}(t)\bigr)
\]
is an interlacing sequence.  Consequently,
\[
\sum_{\wvec\in\Smir_{r,m}}t^{\des(\wvec)}
\]
is real-rooted.
\end{theorem}

\begin{proof}
For $r=1$, every coordinate is the constant polynomial $1$.  Positive
constants are real-rooted polynomials with empty zero sets, so this vector is
interlacing in this convention.
Equation~\eqref{eq:lastLetterRecurrence} and Leander's preservation theorem
give the first claim by induction.

Every nonzero nonnegative linear combination of an interlacing sequence is
real-rooted.  Taking the sum of the last-letter vector gives the unrefined
Smirnov descent polynomial, so the second claim follows.
\end{proof}

\begin{remark}
For $m=1$, the recurrence in
Proposition~\ref{prop:lastLetterRecurrence} gives
$F^{(1)}_{1,1}(t)=1$ and $F^{(1)}_{r,1}(t)=0$ for $r\geq2$.
Theorem~\ref{thm:SmirnovRealRooted} assumes $m\geq2$ so that the zero
polynomial requires no separate convention.
\end{remark}

\begin{proof}[Proof of Theorem~\ref{thm:mainChow}]
By Theorem~\ref{thm:SmirnovTransfer}
and~\cite[Theorem~1.1]{ChunHayashidaPartidaWang2026},
\[
H_{\NC_{n+1}}(t)
=\frac{1}{n+1}\sum_{\wvec\in\Smir_{n,n+1}}t^{\des(\wvec)}.
\]
The polynomial on the right is real-rooted by
Theorem~\ref{thm:SmirnovRealRooted}; alternatively, this scalar statement
follows from the specialization of Br\"and\'en and Vecchi described above.
Multiplication by the positive scalar $1/(n+1)$ does not change its zeros.
The desired result follows.
\end{proof}

\subsection{Coefficient tables}\label{subs:ChowTables}

We set
\[
\defin{C_n(t)}\coloneqq H_{\NC_{n+1}}(t)
=\sum_{k=0}^{n-1}c_{n,k}t^k.
\]
Table~\ref{tab:ChowTriangle} gives the first nine coefficient rows.  Blank
positions are outside the support.  The final column is the row sum, which is
$n^{n-1}$ by Schumacher's tieless enumeration and is recorded as
OEIS~\oeis{A000169}.  At the time of writing, the full coefficient triangle is
not listed in OEIS~\cite{OEIS2026}.

\begin{table}[ht]
\centering
\caption{The coefficient triangle $c_{n,k}=[t^k]C_n(t)$.}
\label{tab:ChowTriangle}
\scriptsize
\resizebox{\textwidth}{!}{%
\begin{tabular}{c*{9}{r}r}
\toprule
$n$&$c_{n,0}$&$c_{n,1}$&$c_{n,2}$&$c_{n,3}$&$c_{n,4}$
&$c_{n,5}$&$c_{n,6}$&$c_{n,7}$&$c_{n,8}$&row sum\\
\midrule
1&1&&&&&&&&&1\\
2&1&1&&&&&&&&2\\
3&1&7&1&&&&&&&9\\
4&1&31&31&1&&&&&&64\\
5&1&116&391&116&1&&&&&625\\
6&1&407&3480&3480&407&1&&&&7776\\
7&1&1401&26097&62651&26097&1401&1&&&117649\\
8&1&4825&178621&865129&865129&178621&4825&1&&2097152\\
9&1&16750&1162684&10228978&20229895&10228978&1162684
&16750&1&43046721\\
\bottomrule
\end{tabular}}
\end{table}

Table~\ref{tab:ChowColumns} records the first two nonconstant columns, a
middle coefficient, and the row sum through $n=15$.
For even $n$, the two middle coefficients agree by palindromicity.  The first
nonconstant column satisfies
\[
c_{n,1}=T(n+1,3),
\]
where $T(n,k)$ is the triangle of reversed partial sums of Narayana rows in
OEIS~\oeis{A104710}; the equality also follows from the closed formula
in~\cite[Corollary~1.2]{ChunHayashidaPartidaWang2026}.  We found no separate OEIS
entries for the second nonconstant column or the middle
coefficients~\cite{OEIS2026}.

\begin{table}[ht]
\centering
\caption{Selected columns of the Chow coefficient triangle.}
\label{tab:ChowColumns}
\begin{tabular}{crrrr}
\toprule
$n$&$c_{n,1}$&$c_{n,2}$&$c_{n,\lfloor(n-1)/2\rfloor}$&$n^{n-1}$\\
\midrule
2&1&--&1&2\\
3&7&1&7&9\\
4&31&31&31&64\\
5&116&391&391&625\\
6&407&3480&3480&7776\\
7&1401&26097&62651&117649\\
8&4825&178621&865129&2097152\\
9&16750&1162684&20229895&43046721\\
10&58730&7362059&382786372&1000000000\\
11&207945&45938080&11015084059&25937424601\\
12&742821&284637432&265930007604&743008370688\\
13&2674348&1759251391&9089277789489&23298085122481\\
14&9694739&10875544864&267453721132672&793714773254144\\
15&35357549&67350530969&10588539011709315&29192926025390625\\
\bottomrule
\end{tabular}
\end{table}

Chun--Hayashida--Partida--Wang proved that $C_n(t)$ is
$\gamma$-positive~\cite[Theorem~1.5]{ChunHayashidaPartidaWang2026}.  We write its
$\gamma$-expansion and corresponding $\gamma$-polynomial as
\begin{align*}
C_n(t)&=\sum_{j=0}^{\lfloor(n-1)/2\rfloor}
\gamma_{n,j}t^j(1+t)^{n-1-2j},\\
\defin{\Gamma_n(t)}&\coloneqq
\sum_{j=0}^{\lfloor(n-1)/2\rfloor}\gamma_{n,j}t^j.
\end{align*}
\begin{corollary}\label{cor:ChowGammaRealRooted}
For every $n\geq1$, the polynomial $\Gamma_n(t)$ is real-rooted.  Moreover,
all its zeros are nonpositive.
\end{corollary}

\begin{proof}
Theorem~\ref{thm:mainChow} and the real-rootedness equivalence for a
palindromic polynomial and its
$\gamma$-polynomial~\cite[Observation~4.2]{Petersen2015} show that
$\Gamma_n(t)$ is real-rooted.
Its coefficients are nonnegative by the $\gamma$-positivity of $C_n(t)$, so
it has no positive zeros.
\end{proof}

The first ten $\gamma$-rows are shown in
Table~\ref{tab:ChowGammaTriangle}.  At the time of writing, neither this
triangle nor its first nonconstant columns appear in OEIS~\cite{OEIS2026}.

\begin{table}[ht]
\centering
\caption{The $\gamma$-coefficients $\gamma_{n,j}$ of $C_n(t)$.}
\label{tab:ChowGammaTriangle}
\begin{tabular}{crrrrr}
\toprule
$n$&$\gamma_{n,0}$&$\gamma_{n,1}$&$\gamma_{n,2}$
&$\gamma_{n,3}$&$\gamma_{n,4}$\\
\midrule
1&1&&&&\\
2&1&&&&\\
3&1&5&&&\\
4&1&28&&&\\
5&1&112&161&&\\
6&1&402&2264&&\\
7&1&1395&20502&13257&\\
8&1&4818&154510&353384&\\
9&1&16742&1062204&5728976&2063809\\
10&1&58721&6950976&73804733&89807052\\
\bottomrule
\end{tabular}
\end{table}

\section{A peak--tie insertion recurrence}\label{sec:peakTie}

We now turn from descents to peaks.  The content transfer of
Section~\ref{sec:transfer} is not available by
Remark~\ref{rem:transferLimitation}, so we work directly with bounded-maximum
parking functions.

For $1\leq m\leq n$, set
\[
\defin{\PF_{n,m}}
\coloneqq
\{\wvec\in\PF_n:\max_i w_i\leq m\}.
\]
Thus $H_{n,m}(t,u)$ is the peak--tie polynomial of $\PF_{n,m}$, and
$H_{n,n}(t,u)$ is the full peak--tie polynomial of $\PF_n$.

Let $\wvec=w_1\dotsm w_\ell$ be a word, and let $M$ be larger than every
letter of $\wvec$.  We insert copies of $M$ in the $\ell+1$ gaps of $\wvec$,
allowing several copies in the same gap.  Introduce a variable $z$ for the
number of inserted copies and set
\begin{equation}\label{eq:gapFactors}
\begin{aligned}
\defin{\mathcal E}&\coloneqq1+\frac{z}{1-uz},\\
\defin{\mathcal A}&\coloneqq1+tz+\frac{uz^2}{1-uz},\\
\defin{\mathcal C}&\coloneqq u+tz+\frac{uz^2}{1-uz},\\
\defin{\mathcal B}&\coloneqq \mathcal A^2+t-1.
\end{aligned}
\end{equation}
These four generating series correspond to endpoint gaps, ordinary internal
gaps, tie gaps, and the paired gaps adjacent to an existing peak.

For fixed $\ell\geq1$ and $r\geq0$, define the linear
\defin{insertion operator}
$\mathcal I_{\ell,r}\colon\mathbb R[t,u]\to\mathbb R[t,u]$ by
\begin{equation}\label{eq:insertionOperator}
\defin{\bigl(\mathcal I_{\ell,r}F\bigr)(t,u)}
\coloneqq
[z^r]\mathcal E^2\mathcal A^{\ell-1}
F\!\left(\frac{\mathcal B}{\mathcal A^2},
\frac{\mathcal C}{\mathcal A}\right).
\end{equation}
\begin{lemma}\label{lem:insertionOperator}
Let $F(t,u)$ enumerate a family of words of length $\ell$ by peaks and
ties.  Then $\bigl(\mathcal I_{\ell,r}F\bigr)(t,u)$ is the polynomial obtained
by inserting exactly $r$ copies of a new maximal letter in every word.
\end{lemma}

\begin{proof}
Suppose that $\wvec$ has $p$ peaks and $k$ ties.  Inserting $s\geq1$
copies at an endpoint creates $s-1$ ties and no peak, which gives
$\mathcal E$.

At an internal gap not adjacent to a peak, one inserted copy creates a
new peak.  A block of $s\geq2$ copies creates $s-1$ ties and no
peak.  This gives $\mathcal A$.  At an existing tie, inserting nothing
preserves its factor $u$, one copy destroys the tie and creates a peak, and
$s\geq2$ copies replace the old tie by $s-1$ ties.  This gives $\mathcal C$.

Two peaks cannot occur in adjacent positions.  Hence the two gaps adjacent
to one peak are distinct from all gaps adjacent to every other peak.  If
neither gap is occupied, the old peak survives.  If at least one is occupied,
the old peak disappears, and the two $\mathcal A$ factors record the new local
peaks and ties.  The joint factor is therefore
$\mathcal A^2+t-1=\mathcal B$.

No other interaction joins two gaps.  A new peak is centered at a single
inserted copy in an internal gap, an old tie changes only when its own gap is
occupied, and an old nonpeak cannot become a peak after a larger neighbor is
inserted.  Thus the four classes above exhaust all changes in peaks and ties.

There are $\ell-1-2p-k$ remaining nontie internal gaps.  Hence all insertions
into $\wvec$ have generating series
\[
\mathcal E^2\mathcal A^{\ell-1-2p-k}\mathcal B^p\mathcal C^k.
\]
The monomial $t^pu^k$ in $F$ is transformed into this expression by the
substitution in~\eqref{eq:insertionOperator}.  Extracting $[z^r]$ completes
the proof.
\end{proof}

\begin{proof}[Proof of Theorem~\ref{thm:parkingPeakTieRecurrence}]
We partition $\PF_{n,m}$ according to the number $r$ of occurrences of $m$.
The case $r=0$ contributes $H_{n,m-1}(t,u)$.

Suppose that $r>0$.  In the nondecreasing rearrangement, the first copy of
$m$ is in position $n-r+1$.  The parking condition therefore gives
$m\leq n-r+1$, or equivalently $r\leq n-m+1$.

Deleting all copies of $m$ leaves an element of $\PF_{n-r,m-1}$.
Conversely, the inequality $n-r\geq m-1$ shows that inserting $r$ copies of
$m$ into any element of $\PF_{n-r,m-1}$ produces a parking function.  More
precisely, deletion gives the shorter word together with the unique block
size of the deleted letters in each of its gaps.  These nonnegative block
sizes sum to $r$, and reinserting the corresponding blocks recovers the
original word.  We therefore have a bijection between the words in question
and the insertions enumerated by Lemma~\ref{lem:insertionOperator}, which
gives~\eqref{eq:parkingPeakTieRecurrence}.

Finally, the only word in $\PF_{n,1}$ is $1^n$.  It has $n-1$ ties and no
peak, which proves~\eqref{eq:parkingPeakTieInitial}.  This completes
the proof.
\end{proof}

Theorem~\ref{thm:parkingPeakTieRecurrence} determines every number
\[
\#\{\wvec\in\PF_n:\pk(\wvec)=j,\ \tie(\wvec)=k\}.
\]
It therefore supplies a recursive answer to Problems~3 and~4 of
Cruz et al.~\cite{CruzHarrisHarryEtAl2024}.

\subsection{Peaks alone}\label{subs:peaksAlone}

Set
\[
\defin{P_{n,m}(t)}\coloneqq H_{n,m}(t,1),
\qquad
\defin{P_n(t)}\coloneqq P_{n,n}(t).
\]
When $u=1$, the factors in~\eqref{eq:gapFactors} become
\[
\mathcal E=\frac{1}{1-z},
\qquad
\mathcal A=1+tz+\frac{z^2}{1-z},
\qquad
\mathcal B=\mathcal A^2+t-1.
\]
Thus the peak insertion operator is
\begin{equation}\label{eq:peakInsertionOperator}
\bigl(\mathcal I_{\ell,r}F\bigr)(t)
=[z^r]\frac{\mathcal A^{\ell-1}}{(1-z)^2}
F\!\left(\frac{\mathcal B}{\mathcal A^2}\right).
\end{equation}
For one inserted maximum, the operator has a familiar differential form.

\begin{proposition}\label{prop:singleMaximumOperator}
For a polynomial $F(t)$ enumerating words of length $\ell$ by peaks,
\begin{equation}\label{eq:singleMaximumOperator}
\bigl(\mathcal I_{\ell,1}F\bigr)(t)
=\bigl(2+(\ell-1)t\bigr)F(t)
+2t(1-t)F'(t).
\end{equation}
\end{proposition}

\begin{proof}
Consider a word with $p$ peaks.  Insertion in either endpoint gap
preserves $p$.  The $2p$ gaps adjacent to existing peaks replace an old peak
by the new maximal peak and also preserve $p$.  Each of the remaining
$\ell-1-2p$ gaps creates one additional peak.  Hence $t^p$ is sent to
\[
(2+2p)t^p+(\ell-1-2p)t^{p+1}.
\]
Extending linearly and collecting the terms involving $p$ gives
\eqref{eq:singleMaximumOperator}, as desired.
\end{proof}

Since the value $n$ occurs at most once in a parking function of length $n$,
Theorem~\ref{thm:parkingPeakTieRecurrence} and
Proposition~\ref{prop:singleMaximumOperator} give
\begin{equation}\label{eq:topPeakRecurrence}
P_n(t)
=P_{n,n-1}(t)
+\bigl(2+(n-2)t\bigr)P_{n-1}(t)
+2t(1-t)P'_{n-1}(t).
\end{equation}
The differential peak operator is therefore present, but the bounded-maximum
term $P_{n,n-1}(t)$ obstructs a recurrence in $P_n$ and $P_{n-1}$ alone.

Peaklessness alone does not lead to a real-rooted descent refinement.  For
instance,
\[
\sum_{\substack{\wvec\in\PF_5\\\pk(\wvec)=0}}t^{\des(\wvec)}
=42+169t+119t^2+20t^3+t^4,
\]
is not real-rooted; its discriminant is $-292515183$.

\section{The Narayana refinement}\label{sec:narayana}

We now give a descent refinement of the Catalan result for
peakless-tieless parking functions.

Let $\PF_n^\uparrow$ denote the set of nondecreasing parking functions.
For $\wvec\in\PF_n^\uparrow$, let
\[
\defin{\Tie^\uparrow(\wvec)}
\coloneqq\{r\in\{2,\dotsc,n\}:w_{r-1}=w_r\}
=\{t_1<t_2<\dotsb<t_j\},
\]
and list the distinct values of $\wvec$ as
\[
1=b_1<b_2<\dotsb<b_k.
\]
Cruz et al. define the following map from the nondecreasing model to the
peakless-tieless model:
\begin{equation}\label{eq:CruzBijection}
\begin{aligned}
\defin{\Psi_n}:\PF_n^\uparrow
&\longrightarrow
\{\mathbf v\in\PF_n:\pk(\mathbf v)=\tie(\mathbf v)=0\},\\
\wvec
&\longmapsto(t_j,t_{j-1},\dotsc,t_1,1,b_2,\dotsc,b_k).
\end{aligned}
\end{equation}
They prove that it is a bijection~\cite[Theorem~4.4]
{CruzHarrisHarryEtAl2024}.  The arrow in~\eqref{eq:CruzBijection}
starts with a nondecreasing parking function and produces a word that is not
necessarily nondecreasing.  Its inverse is also concrete.  A peakless-tieless
word decreases to its unique entry $1$ and then increases.  The entries
before $1$, read increasingly, give the tie positions $t_1,\dotsc,t_j$;
the entries from $1$ onward give the distinct values
$1,b_2,\dotsc,b_k$.  These two lists determine the original nondecreasing
parking function: use the values $b_1,\dotsc,b_k$ as its successive constant
runs and, among positions $2,\dotsc,n$, start a new run precisely at those
not among $t_1,\dotsc,t_j$.

For example, if
\[
\wvec=(1,1,2,2,4),
\]
then its tie positions are $2,4$ and its distinct values are $1,2,4$, so
\[
\Psi_5(1,1,2,2,4)=(4,2,1,2,4).
\]
Splitting the word on the right at $1$ recovers the two lists and hence
recovers $\wvec$.

There is a second, classical bijection behind the Narayana numbers.  Given
$\wvec\in\PF_n^\uparrow$, set
\begin{equation}\label{eq:parkingAreaSequence}
\defin{a_r}\coloneqq r-w_r,
\qquad 1\leq r\leq n.
\end{equation}
The parking-function inequalities and monotonicity say precisely that
$a_1=0$, $a_r\geq0$, and $a_{r+1}\leq a_r+1$.  Thus
$(a_1,\dotsc,a_n)$ is the area sequence of a Dyck path.  Conversely,
$w_r=r-a_r$ recovers the nondecreasing parking function.  Moreover,
\begin{equation}\label{eq:tiesAndDoubleRises}
w_{r-1}=w_r
\quad\Longleftrightarrow\quad
a_r=a_{r-1}+1.
\end{equation}
On the Dyck path, the condition on the right means that the $r$th north step
immediately follows the $(r-1)$st north step.  Hence ties correspond to
double rises.  Every vertical run ends in one Dyck-path peak, that is, an
occurrence $NE$, and a run of length $s$ contributes $s-1$ double rises.
Consequently, $j$ ties correspond to $n-j$ Dyck-path peaks.  Here a
Dyck-path peak is distinct from the word statistic $\pk$.
For the running example, the area sequence is $(0,1,1,2,1)$; its two rises
by one occur at positions $2$ and $4$, so the corresponding Dyck path has
two double rises and three peaks.

\begin{theorem}\label{thm:NarayanaRefinement}
For every $n\geq1$,
\begin{equation}\label{eq:NarayanaRefinement}
\sum_{\substack{\wvec\in\PF_n\\\pk(\wvec)=0\\\tie(\wvec)=0}}
t^{\des(\wvec)}
=\sum_{j=0}^{n-1}
\frac{1}{n}\binom{n}{j}\binom{n}{j+1}t^j.
\end{equation}
\end{theorem}

\begin{proof}
In~\eqref{eq:CruzBijection}, the segment
$t_j,t_{j-1},\dotsc,t_1,1$ has exactly $j$ descents, while
$1,b_2,\dotsc,b_k$ is strictly increasing.  Consequently,
\[
\des(\Psi_n(\wvec))=\tie(\wvec).
\]
By~\eqref{eq:parkingAreaSequence}--\eqref{eq:tiesAndDoubleRises}, the source
words with $j$ ties correspond to Dyck paths with $n-j$ peaks.  The Narayana
number for these paths is
\[
N(n,n-j)
=\frac{1}{j+1}\binom{n}{j}\binom{n-1}{j}
=\frac{1}{n}\binom{n}{j}\binom{n}{j+1};
\]
this is also Schumacher's tie enumeration~\cite[Theorem~12]{Schumacher2018}.
Summing by $j$ and applying the bijection $\Psi_n$ proves
\eqref{eq:NarayanaRefinement}.
\end{proof}

The right side of~\eqref{eq:NarayanaRefinement} is the ordinary Narayana
polynomial.  In particular, the peakless-tieless descent polynomial is
real-rooted.

\section{Ordinary descents}
\label{sec:ordinaryDescents}

We next record a short observation about ordinary descents.  It is independent
of Theorem~\ref{thm:SmirnovTransfer}: here ties are allowed, and the relevant
word family is $[r]^n$ rather than the Smirnov words.  The result below is a
direct consequence of two known ingredients.

For positive integers $n$ and $r$, define the \defin{all-word descent
polynomial}
\[
\defin{E_{n,r}(t)}
\coloneqq
\sum_{\wvec\in[r]^n}t^{\des(\wvec)}.
\]
Athanasiadis, Douvropoulos, and Kalampogia-Evangelinou prove that
$E_{n,r}(t)$ is real-rooted for all $n,r\geq1$; see
\cite[proof of Theorem~3(a)]
{AthanasiadisDouvropoulosKalampogiaEvangelinou2024}.

\begin{corollary}\label{cor:parkingDescentRealRooted}
For every $n\geq1$,
\begin{equation}\label{eq:DiaconisHicksDescentTransfer}
\defin{D_n(t)}
\coloneqq
\sum_{\wvec\in\PF_n}t^{\des(\wvec)}
=\frac{1}{n+1}\sum_{\wvec\in[n+1]^n}t^{\des(\wvec)}.
\end{equation}
In particular, $D_n(t)$ is real-rooted.  Its coefficients form the rows of
OEIS~\oeis{A333829}~\cite{OEIS2026}.
\end{corollary}

\begin{proof}
For $n\geq2$, Diaconis and Hicks show that the full descent-set distribution
on a uniformly chosen parking function in $\PF_n$ agrees with that on a
uniformly chosen word in $[n+1]^n$~\cite[Theorem~8]{DiaconisHicks2017}.  Thus,
for every $S\subseteq[n-1]$, we have
\[
\frac{|\{\wvec\in\PF_n:\Des(\wvec)=S\}|}{(n+1)^{n-1}}
=
\frac{|\{\wvec\in[n+1]^n:\Des(\wvec)=S\}|}{(n+1)^n}.
\]
Multiplying by $t^{|S|}$ and summing over $S\subseteq[n-1]$ gives
\eqref{eq:DiaconisHicksDescentTransfer}; the case $n=1$ follows directly.
The sum on its right-hand side is $E_{n,n+1}(t)$, which is real-rooted by the
all-word result above.  Since multiplication by the positive scalar
$1/(n+1)$ does not change the zeros, $D_n(t)$ is real-rooted.
\end{proof}

\section{Toric \texorpdfstring{$g$}{g}-contributions and 123-avoidance}
\label{sec:toricContributions}

We next study the universal contribution polynomials that occur in the toric
$g$-polynomial of a simple polytope.  Xiao proved their
real-rootedness and two adjacent-rank interlacings by a differential
three-term recurrence and a Liu--Wang-type sign argument
\cite[Theorem~1.3]{Xiao2026}.  At the suggestion of C.~Athanasiadis, Xiao
conjectured that every fixed row is an interlacing sequence.  We prove that
conjecture and the common-interlacer statement needed for 123-avoiding parking
functions; individual real-rootedness follows as a corollary.
Appendix~\ref{app:independentToricRealRootedness} gives a finite-convolution
proof of that corollary which does not use the fixed-row theorem.

Let $P$ be an $n$-dimensional simple polytope, and write
\[
f(P,t)=h_0+h_1t+\dotsb+h_nt^n
\]
for the recursively defined toric $h$-polynomial of its face lattice.  With
$m=\lfloor n/2\rfloor$, Stanley's \defin{toric $g$-polynomial} is
\begin{equation}\label{eq:toricGDefinition}
\defin{g(P,t)}\coloneqq
h_0+\sum_{i=1}^{m}(h_i-h_{i-1})t^i.
\end{equation}
The toric $h$-vector in this definition is distinct from the ordinary
$h$-vector used to define the $\gamma$-vector of a simple polytope; see
\cite[Section~2.1]{EhrenborgHetyeiReaddy2025}.

Let $\Cat_r=\frac{1}{r+1}\binom{2r}{r}$ be the $r$th Catalan number.  Ehrenborg,
Hetyei, and Readdy define the \defin{toric $g$-contribution polynomials}
\begin{equation}\label{eq:toricContribution}
\defin{g_{n,j}(t)}
\coloneqq
\sum_{k=0}^{\lfloor n/2\rfloor}
\Cat_{n-k-j}\binom{n-k}{k}(t-1)^k,
\qquad 0\leq j\leq\lfloor n/2\rfloor.
\end{equation}
If $(\gamma_0,\dotsc,\gamma_{\lfloor n/2\rfloor})$ is the ordinary
$\gamma$-vector of $P$, then
\begin{equation}\label{eq:EhrToricDecomposition}
g(P,t)=\sum_{j=0}^{\lfloor n/2\rfloor}\gamma_jg_{n,j}(t);
\end{equation}
see~\cite[Theorem~3.4]{EhrenborgHetyeiReaddy2025}.

We use one reciprocal normalization throughout this section.  For \(n\geq1\),
write
\[
n=2m+\epsilon,\qquad \epsilon\in\{0,1\},\qquad
c=\epsilon+\frac12,
\]
and use the reverse offset \(d=m-j\).  For \(0\leq d\leq m\), coefficient
comparison in~\eqref{eq:toricContribution} gives
\begin{equation}\label{eq:RdDefinition}
\defin{R_d(x)}
=\frac{(-1)^mx^m}{\binom{m+\epsilon}{m}\Cat_{\epsilon+d}}
g_{n,m-d}\left(1-\frac1x\right)
={}_3F_2\!\left(
\begin{matrix}-m,m+1+\epsilon,c+d\\
c,c+d+\frac32\end{matrix};x\right).
\end{equation}
In particular, \(R_d(0)=1\).  The change of variables \(t=1-1/x\) is
increasing on \((0,1]\), so it preserves root alternation.  We now prove the
theorem directly in these reciprocal coordinates.

\begin{theorem}\label{thm:toricContributionInterlacing}
Let \(n\geq2\) and \(m=\lfloor n/2\rfloor\).  The polynomial
\(g_{n,m}(t)/t\) has degree $m-1$ and satisfies
\(g_{n,m}(t)/t\lless g_{n,j}(t)\) for every
\(0\leq j\leq m\).  Moreover,
\[
\bigl(g_{n,0}(t),g_{n,1}(t),\dotsc,g_{n,m}(t)\bigr)
\]
is an interlacing sequence.  Consequently, every nonzero polynomial
\[
\sum_{j=0}^m\lambda_jg_{n,j}(t),\qquad \lambda_j\geq0,
\]
is real-rooted.
\end{theorem}

\begin{proof}
The case \(m=1\) follows directly from~\eqref{eq:toricContribution}, so assume
\(m\geq2\).  We first compare all but the final contribution by a triangular
differential refinement.  The final contribution is a boundary Euler inverse
and requires the signed-moment argument for an arbitrary real pencil.  Retain
\(\epsilon,c\) from the normalization preceding~\eqref{eq:RdDefinition}, put
\(N=m-1\).  We use Euler's hypergeometric transformation
\[
{}_2F_1(a,b;q;x)
=(1-x)^{q-a-b}{}_2F_1(q-a,q-b;q;x).
\]
After cancellation at \(d=0\) in~\eqref{eq:RdDefinition}, take
\(a=-m\), \(b=m+1+\epsilon\), and \(q=2+\epsilon\).  The exponent in Euler's
transformation is then one, and the identity gives
\begin{equation}\label{eq:S0Definition}
\defin{S_0(x)}\coloneqq\frac{R_0(x)}{1-x}
={}_2F_1\!\left(
\begin{matrix}1-m,m+2+\epsilon\\2+\epsilon\end{matrix};x\right)
\doteq \JacP_N^{(c+1/2,1)}(1-2x),
\end{equation}
where \(\JacP_r^{(a,b)}\) is the Jacobi polynomial normalized by
\(\JacP_r^{(a,b)}(1)=\binom{r+a}{r}\), and \(\doteq\) denotes equality up to a
positive scalar.  In particular, \(S_0\) has \(N\) simple zeros in
\((0,1)\).  Moreover,~\eqref{eq:RdDefinition} gives
\[
S_0(x)=
\frac{(-1)^{m+1}x^{m-1}}
{\binom{m+\epsilon}{m}\Cat_\epsilon}
\frac{g_{n,m}(1-\frac1x)}{1-\frac1x}.
\]
Thus $R_0$, after removing its boundary factor $1-x$, plays exactly the
role of $g_{n,m}(t)/t$.  The interlacing assertion of the theorem is
equivalent to showing that $S_0$ interlaces every $R_d$.

\emph{Finite offsets.}
We first treat $0\leq d\leq N$.  The triangular refinement below starts row
$d$ with the $d$th derivative of $S_0$, of degree $N-d$.  Moving right
applies first-order operators, each of which inserts one zero in $(0,1)$, so
the diagonal again has degree $N$.  The resulting diagonal entry is a positive
multiple of $S_d$.  Comparing the horizontal and up-right maps at the right
boundary then orders the diagonal zeros, allowing us to compare every $S_d$
with $S_0$.

Put $\theta=x\,d/dx$ and
$(c+\theta)_r=(c+\theta+r-1)\dotsm(c+\theta)$.  Write
$R_d=(1-x)S_d$ and, for $0\leq r\leq d$, define
\begin{equation}\label{eq:TdrDefinition}
\defin{T_{d,r}(x)}
\coloneqq(-1)^d(1-x)^{-(d+1-r)}(c+\theta)_r
\left((1-x)^{d+1}S_0^{(d)}(x)\right).
\end{equation}
For example, the first three rows, when present, have the form
\[
\begin{array}{ccccc}
&&&& T_{0,0}=S_0\rule{0pt}{3ex}\\[-1ex]
&&&\nearrow&\\[-1ex]
&&T_{1,0}&\longrightarrow&T_{1,1}\doteq S_1\\[-1ex]
&\nearrow&&\nearrow&\\[-1ex]
T_{2,0}&\longrightarrow&T_{2,1}
&\longrightarrow&T_{2,2}\doteq S_2.
\end{array}
\]
The display is right-justified: $d$ is constant along a horizontal row,
$r=0$ is the left sloping boundary, and $r=d$ is the right vertical boundary.
For fixed $r$, the entries run up-right.  The two boundaries satisfy
$T_{d,0}=(-1)^dS_0^{(d)}$ and $T_{d,d}\doteq S_d$.  Each horizontal or
up-right arrow represents an interval-insertion operator: it raises the
degree by one, inserts one zero in $(0,1)$, and gives a strict interlacing
between its endpoints.  Thus every row and every fixed-$r$ up-right chain is
an adjacent interlacing chain.  The calculation below verifies the arrows and
the right-boundary identification; the subsequent boundary comparison orders
the equal-degree polynomials on that right boundary.

Commuting one Euler factor through the power of $1-x$ gives
\begin{equation}\label{eq:TdrHorizontal}
T_{d,r+1}
=x(1-x)T_{d,r}'+\bigl(c+r-(c+d+1)x\bigr)T_{d,r}.
\end{equation}
If \(f(0)>0\) has simple zeros \(0<r_1<\dotsb<r_s<1\), then
\begin{equation}\label{eq:intervalInsertionOperator}
\defin{\mathcal M_{a,b}f}\coloneqq x(1-x)f'+(a-bx)f,
\qquad a>0,\quad b-a>0,
\end{equation}
has one simple zero in each of
\((0,r_1),(r_1,r_2),\dotsc,(r_s,1)\).  Indeed, its signs at \(0\), at the
zeros of \(f\), and at \(1\) alternate, and its degree is \(s+1\).
Rolle's theorem first gives $T_{d,0}=(-1)^dS_0^{(d)}$ exactly $N-d$ simple
zeros in $(0,1)$, while~\eqref{eq:S0Definition} gives
$T_{d,0}(0)>0$.  Applying~\eqref{eq:TdrHorizontal} successively
then shows that every $T_{d,r}$ has simple zeros in $(0,1)$ and
is positive at zero.

To identify the diagonal, differentiation of~\eqref{eq:S0Definition},
followed by Euler's transformation, gives
\[
(-1)^d(1-x)^{d+1}S_0^{(d)}(x)
\doteq {}_2F_1\!\left(
\begin{matrix}-m,m+1+\epsilon\\\epsilon+2+d\end{matrix};x\right).
\]
Indeed, the omitted factor is
\[
(-1)^d\frac{(1-m)_d(m+\epsilon+2)_d}{(\epsilon+2)_d}>0.
\]
On the coefficient of $x^k$, the normalized Euler operator contributes
\[
\frac{(c+k)_d}{(c)_d}=\frac{(c+d)_k}{(c)_k}.
\]
Comparing with~\eqref{eq:RdDefinition} now gives the only normalization that
we need:
\[
T_{d,d}\doteq S_d.
\]

It remains to compare the diagonals of this triangular array.  For \(d<N\),
set
\[
\Lambda_d=(N-d)\left(N+d+c+\frac52\right),\quad
A_d=c+d+\frac32,\quad B_{d,r}=c+2d+\frac72-r.
\]
The Jacobi differential equation and the intertwining identity
\[
\mathcal M_{a,b}\mathcal M_{A,B}
=\mathcal M_{A,B-1}\mathcal M_{a,b+1},
\]
valid when $A+b+1=a+B$, with
$a=c+r$, $b=c+d+1$, $A=c+d+\frac32$, and
$B=c+2d+\frac72-r$, give the vertical relation
\begin{equation}\label{eq:TdrVertical}
T_{d,r}=\frac1{\Lambda_d}
\left(x(1-x)T_{d+1,r}'+(A_d-B_{d,r}x)T_{d+1,r}\right).
\end{equation}
Since $\Lambda_d>0$ and $B_{d,r}-A_d=d+2-r>0$, this verifies the up-right
arrows in the diagram.

At the right boundary $r=d$ of the triangular array, put
\[
H=T_{d+1,d}.
\]
The vertical and horizontal relations give
\[
\Lambda_dT_{d,d}=\mathcal M_{c+d+3/2,c+d+7/2}H,
\qquad
T_{d+1,d+1}=\mathcal M_{c+d,c+d+2}H.
\]
Subtracting the first identity from the second gives
\begin{equation}\label{eq:BezoutEdge}
T_{d+1,d+1}-\Lambda_dT_{d,d}=\frac32(x-1)H.
\end{equation}
Both polynomials on the left are interval-insertion transforms of \(H\).  At
the left endpoint of each gap of \(H\), the polynomial \(T_{d+1,d+1}\) has
the sign taken by \(H\) inside the gap.  Indeed, this follows from
$T_{d+1,d+1}(0)=(c+d)H(0)$ for the first gap and from
$T_{d+1,d+1}(\rho)=\rho(1-\rho)H'(\rho)$ at a zero $\rho$ of $H$.  At a zero
$a$ of $T_{d,d}$, equation~\eqref{eq:BezoutEdge} gives
$T_{d+1,d+1}(a)=\frac32(a-1)H(a)$, which has the opposite sign.  The unique
zero of $T_{d+1,d+1}$ in that gap therefore lies to the left of $a$.  Thus, if
\(s_{d,1}<\dotsb<s_{d,N}\) are the zeros of \(S_d\), then
\begin{equation}\label{eq:SdGapChain}
s_{d+1,i}<s_{d,i}<s_{d+1,i+1}\quad(1\leq i<N),
\qquad s_{d+1,N}<s_{d,N}.
\end{equation}
At the other endpoint,
\[
S_N\doteq \JacP_N^{(c-1,1)}(1-2x),\qquad
S_0\doteq \JacP_N^{(c+1/2,1)}(1-2x).
\]
The same-degree Jacobi comparison in
\cite[Theorem~2.4]{DriverJordaanMbuyi2008}, after using Jacobi symmetry,
applies because the parameter gap is \(3/2<2\).  Thus, if
$\xi_1<\dotsb<\xi_N$ are the zeros of $S_0$, then
\[
s_{N,i}<\xi_i<s_{N,i+1}\quad(1\leq i<N),
\qquad s_{N,N}<\xi_N<1.
\]
Starting from $s_{0,i}=\xi_i$ and using~\eqref{eq:SdGapChain} forward from
$d=0$ and backward from $d=N$ now gives
\[
s_{d,i}\leq\xi_i\leq s_{d,i+1}\quad(1\leq i<N),
\qquad s_{d,N}\leq\xi_N<1,
\]
for $0\leq d\leq N$.  Hence $S_0$ interlaces every
$R_d=(1-x)S_d$ with $d<m$.

\emph{The exceptional offset.}
The remaining value $d=m$ uses the continuous family $E_\gamma$ and the
signed-moment and endpoint identities
\eqref{eq:boundarySignedMoment}--\eqref{eq:boundaryEulerEndpoint} from
Lemma~\ref{lem:boundaryEulerInverse}, which is stated and proved later in
Appendix~\ref{app:boundaryEulerInverse}.  Put
\[
\alpha=c+m-1,\qquad \beta=\alpha+\frac32,
\]
and set
\[
\defin{\Phi(x)}\coloneqq{}_2F_1(-m,c+m;c;x).
\]
A coefficient comparison gives
\begin{equation}\label{eq:EulerInverseEndpoints}
E_\alpha=R_{m-1}=(1-x)S_N,\qquad E_\beta=R_m.
\end{equation}

The coefficient identity
\[
(\theta+\gamma)E_\gamma=\gamma\Phi
\]
explains the reduction to a common Euler-inverse family.  This inverse
differential relation does not by itself compare the roots at the two
parameters \(\alpha\) and \(\beta\).  The signed moments in the boundary lemma
supply that comparison.

For every $\gamma>\alpha$, these identities apply to $E_\gamma$.  Fix
$\gamma_1,\gamma_2>\alpha$.  The case $\gamma_1=\gamma_2$ is immediate, so
assume $\gamma_1<\gamma_2$.  For the real pencil
$Q=\lambda E_{\gamma_1}+\mu E_{\gamma_2}$, put
\[
K(u)=\lambda E_{\gamma_1}(1)u^{c-\gamma_1-1}
+\mu E_{\gamma_2}(1)u^{c-\gamma_2-1}.
\]
Summing~\eqref{eq:boundarySignedMoment} for $\gamma_1$ and $\gamma_2$ gives
\begin{equation}\label{eq:EulerPencilMoment}
\int_0^1x^{c-1}Q(x)p(x)\,dx
=-\int_1^\infty p(u)K(u)\,du
\qquad (\deg p<m).
\end{equation}
The ratio of the two power functions is strictly monotone, so $K$ changes
sign at most once on $(1,\infty)$.  First suppose that $Q(0)>0$ and
$Q(1)\neq0$.  Let $r_1,\dotsc,r_s$ be its sign-changing zeros in $(0,1)$,
and put $p_0(x)=(-1)^s\prod_i(x-r_i)$.  Then $Qp_0\geq0$ on $(0,1)$ and is
positive away from the zeros of $Q$.  Moreover, $K(1)=Q(1)$ has sign
$(-1)^s$, whereas $p_0$ has sign $(-1)^s$ on $(1,\infty)$.  If $K$ has no
sign change, using $p_0$ in~\eqref{eq:EulerPencilMoment} is a contradiction
when $s<m$.  If it changes sign at $\kappa>1$, then
$p_0(u)(\kappa-u)K(u)>0$ away from $\kappa$; using
$p_0(x)(\kappa-x)$ is a contradiction when $s+1<m$.  Thus, under these
generic endpoint assumptions, $Q$ has at least $m-1$ distinct real zeros in
$(0,1)$.

Furthermore,
\[
\frac{[x^m]E_\gamma}{[x^{m-1}]E_\gamma}
=-\frac{c+2m-1}{m(c+m-1)}
\frac{\gamma+m-1}{\gamma+m}
\]
is strictly monotone in \(\gamma\).  Hence a nonzero pencil member has degree
\(m\) or \(m-1\).  Together with the preceding root count, this makes $Q$
real-rooted under the same generic assumptions: in degree \(m\), a nonreal
remaining root would bring its conjugate.  The case $Q(0)<0$ reduces by
scaling.  The cases $Q(0)=0$ or $Q(1)=0$, including the possible degree drop,
follow by perturbation and coefficientwise closedness of real-rooted
polynomials of degree at most $m$.  Thus every nonzero real pencil member is
real-rooted.  The Obreschkoff--Dedieu theorem, in the form recalled
in~\cite[Section~3]{ChudnovskySeymour2007}, now shows that any two
\(E_\gamma\) interlace.

Let \(r_1(\gamma)<\dotsb<r_m(\gamma)\) be the zeros of \(E_\gamma\).  Since
\(E_\gamma(0)=1\),
\begin{equation}\label{eq:EulerRootOrientation}
\sum_{i=1}^m\frac1{r_i(\gamma)}
=\frac{m(c+m)}c\frac{\gamma}{\gamma+1},
\end{equation}
which is strictly increasing.  This fixes the interlacing orientation: the
zeros move to the left as \(\gamma\) increases, since the two possible
alternating orders give opposite inequalities between the sums of reciprocal
roots.  Index the zeros \(\varphi_i,r_i,a_i\) of
\(\Phi,E_\beta,E_\alpha\), respectively, in increasing order.  The
coefficientwise limits
\(E_\gamma\longrightarrow\Phi\) as \(\gamma\longrightarrow\infty\) and
\(E_\gamma\longrightarrow E_\alpha\) as \(\gamma\downarrow\alpha\) give
\begin{equation}\label{eq:EulerRootTrap}
\varphi_i\leq r_i\leq a_i.
\end{equation}
The Jacobi forms of \(\Phi\) and \(S_0\) have degrees \(m\) and \(m-1\) and
parameter pairs \((c-1,0)\) and \((c+1/2,1)\), respectively.  For the zeros
\(\xi_1<\dotsb<\xi_N\) introduced above,
\cite[Theorem~2.3]{DriverJordaanMbuyi2008} gives
\[
\varphi_i<\xi_i<\varphi_{i+1}\qquad(1\leq i<m).
\]
The same-degree endpoint comparison above gives
\(a_i<\xi_i<a_{i+1}\), where \(a_m=1\).  Hence
\[
r_i\leq a_i<\xi_i<\varphi_{i+1}\leq r_{i+1}
\qquad(1\leq i<m),
\]
so \(S_0\) also interlaces \(R_m\).

We finally record the orientations among the $R_d$.  For
$0\leq d<e<m$, repeated use of~\eqref{eq:SdGapChain} gives
$s_{e,i}<s_{d,i}$, while the common comparison with the zeros of $S_0$ gives
\[
s_{d,i}\leq\xi_i\leq s_{e,i+1}\qquad(1\leq i<N).
\]
Together with the common root $1$ of the $R_d$ for $d<m$, these inequalities
give the oriented alternating root pattern with the zeros of $R_e$ first.

For the exceptional offset, let $u_1<\dotsb<u_m=1$ be the zeros of $R_d$,
where $d<m$.  The preceding coordinatewise ordering and
\eqref{eq:EulerRootTrap} give $r_i\leq a_i\leq u_i$.  The two Jacobi
comparisons above give
\[
u_i\leq\xi_i<\varphi_{i+1}\leq r_{i+1}\qquad(1\leq i<m),
\]
and $r_m<1=u_m$.  Hence the zeros of $R_m$ and $R_d$ have the same oriented
alternating pattern for every $d<m$.

The reciprocal identities above and~\eqref{eq:RdDefinition}, together with
the monotonicity of $t=1-1/x$ and the relation $j=m-d$, now translate the
common-interlacer statement and these oriented patterns into
\[
\frac{g_{n,m}(t)}t\lless g_{n,j}(t)\quad(0\leq j\leq m),
\qquad
g_{n,i}(t)\lless g_{n,j}(t)\quad(0\leq i<j\leq m).
\]
The $t^m$-coefficient of $g_{n,j}$ is
$\Cat_{m+\epsilon-j}\binom{m+\epsilon}{m}>0$.  Hence every nonnegative
linear combination is real-rooted; see
\cite[Statement~3.6]{ChudnovskySeymour2007}.  This completes the proof.
\end{proof}

\begin{corollary}\label{cor:toricContributionRealRooted}
For every \(n\geq1\) and
\(0\leq j\leq\lfloor n/2\rfloor\), the toric contribution
\(g_{n,j}(t)\) is real-rooted.
\end{corollary}

\begin{proof}
For \(n\geq2\), this follows from
Theorem~\ref{thm:toricContributionInterlacing}; the case \(n=1\) is immediate
from~\eqref{eq:toricContribution}.
\end{proof}

\begin{remark}\label{rem:XiaoComparison}
Xiao proved
Corollary~\ref{cor:toricContributionRealRooted} by a differential three-term
recurrence derived from the Catalan generating function, initial boundary
interlacing, and a Liu--Wang-type induction.  Xiao also proves
$g_{n,j}\lless g_{n+1,j}$ and $g_{n,j}\lless g_{n+1,j+1}$ in the common
orientation of Xiao's notation and ours~\cite[Theorem~1.3]{Xiao2026}.
Appendix~\ref{app:independentToricRealRootedness} gives a separate proof of
the corollary by finite multiplicative convolution and the boundary
Euler-inverse lemma.

Xiao defines an interlacing sequence by the same oriented all-pairs condition
used in Section~\ref{sec:interlacing}.  Thus the fixed-row assertion of
Theorem~\ref{thm:toricContributionInterlacing}, together with the immediate
cases $n=0,1$, proves Conjecture~4.2 of~\cite{Xiao2026}, which Xiao formulated
at the suggestion of C.~Athanasiadis.  Xiao observed that this conjecture
would imply Conjecture~11.2 of Ehrenborg--Hetyei--Readdy
\cite[Section~4]{Xiao2026}.  By
\cite[Statement~3.6]{ChudnovskySeymour2007}, the common interlacer and the
positive leading coefficients directly give compatibility, hence
real-rootedness of every nonnegative row sum.  They do not by themselves
determine the oriented order of each pair; that is the extra conclusion of
the root-order inequalities in the proof of
Theorem~\ref{thm:toricContributionInterlacing}.
\end{remark}

Following Ehrenborg, Hetyei, and Readdy, let \(\PF_n(123)\) be the set of
parking functions having no \(i<j<k\) with
\(w_i\leq w_j\leq w_k\), and let
\(\defin{\wasc(\wvec)}\coloneqq|\{i:w_i\leq w_{i+1}\}|\).

\begin{corollary}\label{cor:avoidingParkingRealRooted}
Conjecture~11.2 of Ehrenborg--Hetyei--Readdy holds.  Moreover, for every
\(n\geq1\), both
\[
\defin{A_n(t)}\coloneqq
\sum_{\wvec\in\PF_n(123)}t^{\wasc(\wvec)}
\qquad\text{and}\qquad
\sum_{\wvec\in\PF_n(123)}t^{\des(\wvec)}
\]
are real-rooted, where 123-avoidance and ascents are weak.
\end{corollary}

\begin{proof}
The case \(n=1\) is immediate.  For \(n\geq2\),
Theorem~\ref{thm:toricContributionInterlacing} proves that the toric
\(g\)-polynomial of every simple polytope with nonnegative \(\gamma\)-vector
is real-rooted by~\eqref{eq:EhrToricDecomposition}.  The cyclohedron and every
chordal nestohedron have nonnegative \(\gamma\)-vectors by
\cite[Proposition~7.1 and Theorem~10.4]{EhrenborgHetyeiReaddy2025}, so this
gives Conjecture~11.2.

For the associahedron,
\(\gamma_j=\Cat_j\binom{n}{2j}\geq0\), and Ehrenborg, Hetyei, and Readdy
identify its toric $g$-polynomial with $A_n(t)$; see
\cite[Proposition~6.1 and Theorem~6.2]{EhrenborgHetyeiReaddy2025}.  Finally,
$\wasc(\wvec)+\des(\wvec)=n-1$, and therefore the strict-descent polynomial
is $t^{n-1}A_n(1/t)$.  Reciprocal reversal and the possible insertion of
zeros at the origin preserve real-rootedness.  The claim follows.
\end{proof}

The fixed-row interlacing statement has been checked in Lean in the companion
\texttt{RealRooted} library; see
\href{https://github.com/PerAlexandersson/RealRooted/blob/%
acd0ec31118a155b083c8dd45af2015492ce0c10/RealRooted/ParkingFunctions/%
ToricContribution/ContributionReversal.lean\#L420}
{the formalized theorem} in commit \texttt{acd0ec3}.  The same development
checks the common-interlacer theorem and its nonnegative-sum consequence,
including the exceptional Euler-inverse pencil, and proves the coefficient
reversal between the EHR contributions and the normalized hypergeometric
polynomials used in the proof.

\section{Further real-rootedness results}
\label{sec:furtherRealRootedness}

We collect two short consequences concerning other parking-function
statistics.  Both are independent of the preceding toric and Smirnov-word
arguments.

\subsection{Image size}

For $n\geq1$, define the image-size polynomial
\[
\defin{M_n(t)}
\coloneqq
\sum_{\wvec\in\PF_n}t^{|\operatorname{im}(\wvec)|-1}.
\]
Its coefficients form displayed row $n$ of
OEIS~\oeis{A141618}~\cite{OEIS2026}.

We write
$\defin{\left\{\begin{matrix}n\\k\end{matrix}\right\}}$
for the Stirling number of the second kind, the number of partitions of
$[n]$ into $k$ nonempty blocks.

\begin{theorem}\label{thm:imageSizeRealRooted}
For every $n\geq1$,
\begin{equation}\label{eq:imageSizeCoefficients}
M_n(t)
=\sum_{j=0}^{n-1}
\binom{n}{j}
\left\{\begin{matrix}n\\j+1\end{matrix}\right\}j!\,t^j.
\end{equation}
Moreover, $M_n(t)$ has $n-1$ simple zeros, all of which are negative.
\end{theorem}

\begin{proof}
For integers $\ell>r\geq1$, put
\[
\defin{U_{r,\ell}(t)}
\coloneqq\frac{1}{\ell}\sum_{\wvec\in[\ell]^r}
t^{|\operatorname{im}(\wvec)|-1}.
\]
Partitioning the positions into equality fibers and then assigning distinct
alphabet letters gives
\begin{equation}\label{eq:occupancyPolynomial}
U_{r,\ell}(t)
=\sum_{j=0}^{r-1}
(\ell-1)_{\underline{j}}
\left\{\begin{matrix}r\\j+1\end{matrix}\right\}t^j.
\end{equation}
Here
$\defin{(a)_{\underline j}}\coloneqq a(a-1)\dotsm(a-j+1)$.
Pollak's cyclic action adds a common residue modulo $n+1$ to every letter of
a word in $[n+1]^n$.  Every orbit contains a unique parking function, and
image size is constant on orbits; see also
\cite[Section~4]{DiaconisHicks2017}.  Hence
\begin{equation}\label{eq:PollakImageTransfer}
M_n(t)=U_{n,n+1}(t).
\end{equation}
Equation~\eqref{eq:imageSizeCoefficients} now follows from
\eqref{eq:occupancyPolynomial}.

It remains to locate the zeros.  Set
\[
\defin{\widehat U_{r,\ell}(x)}
\coloneqq(1+x)^{r-1}U_{r,\ell}\left(\frac{x}{1+x}\right).
\]
Writing again $\theta=x\,d/dx$, the occupancy expansion is equivalently
\[
\widehat U_{r,\ell}(x)
=\frac{\theta^r(1+x)^\ell}
{\ell x(1+x)^{\ell-r}}.
\]
Applying $\theta$ once more gives, whenever $\ell>r+1$,
\begin{equation}\label{eq:occupancySturmRecurrence}
\widehat U_{r+1,\ell}(x)
=\bigl(1+(\ell-r+1)x\bigr)\widehat U_{r,\ell}(x)
+x(1+x)\widehat U_{r,\ell}'(x).
\end{equation}
We claim by induction that $\widehat U_{r,\ell}$ has $r-1$ simple zeros in
$(-1,0)$.  The assertion starts with $\widehat U_{1,\ell}=1$.  At a zero
$\rho\in(-1,0)$ of $\widehat U_{r,\ell}$, the derivative term in
\eqref{eq:occupancySturmRecurrence} has sign opposite to
$\widehat U_{r,\ell}'(\rho)$.  These signs alternate.  At the endpoints,
\[
\widehat U_{r+1,\ell}(0)=1,
\qquad
\widehat U_{r+1,\ell}(-1)=(r-\ell)\widehat U_{r,\ell}(-1).
\]
Thus one new zero lies in each interval cut out by the old zeros and the two
endpoints.  The transformed version of~\eqref{eq:occupancyPolynomial} has
positive coefficients and degree $r$, so these are all the zeros and they are
simple.

For fixed $n$, apply this induction with $\ell=n+1$ up to $r=n$.
The increasing bijection $x\mapsto x/(1+x)$ sends $(-1,0)$ onto
$(-\infty,0)$, so~\eqref{eq:PollakImageTransfer} transports the claimed
zeros to $M_n$.
\end{proof}

\subsection{Two-pattern avoidance in the labeled-Dyck model}

We next use a notion of pattern avoidance different from the weak word
avoidance in Section~\ref{sec:toricContributions}.  For
$\wvec=(w_1,\dotsc,w_n)\in\PF_n$, let
\[
\defin{B_i(\wvec)}\coloneqq\{j\in[n]:w_j=i\}.
\]
List the elements of $B_1(\wvec)$ increasingly, then those of
$B_2(\wvec)$, and so on, omitting empty blocks.  The resulting permutation
of $[n]$ is denoted by $\defin{\phi(\wvec)}$.  Equivalently, $\phi(\wvec)$ is
obtained by reading the north-step labels in the labeled-Dyck-path model of
$\wvec$.  Following Adeniran and Pudwell~\cite{AdeniranPudwell2023}, we say
that $\wvec$ avoids a permutation pattern if $\phi(\wvec)$ does.

Let
\[
\defin{\mathscr P}
\coloneqq\bigl\{\{132,213\},\{132,312\},
          \{213,231\},\{231,312\}\bigr\}.
\]
For $r\geq0$, define
\begin{equation}\label{eq:polygonDissectionPolynomial}
\defin{V_r(t)}
\coloneqq
\sum_{k=0}^{r}
\frac{1}{k+1}\binom{r}{k}\binom{r+k+2}{k}t^k.
\end{equation}
These are the row polynomials of OEIS~\oeis{A033282}, with displayed row
$r+3$.

\begin{theorem}\label{thm:twoPatternRealRooted}
For every $n\geq1$ and every $\Pi\in\mathscr P$,
\begin{align}
\sum_{\substack{\wvec\in\PF_n\\
                  \phi(\wvec)\text{ avoids }\Pi}}
t^{\asc(\phi(\wvec))}
&=V_{n-1}(t),
\label{eq:twoPatternAscents}\\
\sum_{\substack{\wvec\in\PF_n\\
                  \phi(\wvec)\text{ avoids }\Pi}}
t^{\des(\phi(\wvec))}
&=t^{n-1}V_{n-1}(1/t).
\label{eq:twoPatternDescents}
\end{align}
The coefficients in~\eqref{eq:twoPatternAscents} form displayed row $n+2$
of OEIS~\oeis{A033282}, while those in~\eqref{eq:twoPatternDescents} form
displayed row $n$ of OEIS~\oeis{A126216}.  Both polynomial families have
simple negative zeros, and their consecutive rows strictly interlace.
\end{theorem}

\begin{proof}
The proof of~\cite[Theorem~23]{AdeniranPudwell2023} describes the admissible
labelings of each Dyck path for all four choices of $\Pi$.  If the path has
$j$ vertical runs, then the labels within those runs force $n-j$ ascents.
At each of the remaining $j-1$ boundaries, one chooses independently whether
the adjacent labels form an ascent or a descent.  Since Dyck paths of
semilength $n$ with $j$ vertical runs are counted by the Narayana number
\[
N(n,j)=\frac{1}{n}\binom{n}{j}\binom{n}{j-1},
\]
the ascent polynomial is
\begin{equation}\label{eq:twoPatternNarayanaTransform}
\sum_{j=1}^{n}N(n,j)t^{n-j}(1+t)^{j-1}.
\end{equation}
Extracting the coefficient of $t^k$ and applying the Chu--Vandermonde
identity gives
\[
[t^k]\sum_{j=1}^{n}N(n,j)t^{n-j}(1+t)^{j-1}
=\frac{1}{k+1}\binom{n-1}{k}\binom{n+k+1}{k}.
\]
This proves~\eqref{eq:twoPatternAscents}.  Since $\phi(\wvec)$ is a
permutation of length $n$, its numbers of ascents and descents sum to $n-1$;
hence~\eqref{eq:twoPatternDescents} follows by reciprocal reversal.

It remains to prove the root statements.  Coefficient comparison with the
hypergeometric form of the Jacobi polynomial gives
\[
V_r(t)\doteq \JacP_r^{(1,1)}(1+2t).
\]
The zeros of consecutive Jacobi polynomials $\JacP_r^{(1,1)}$ are simple,
lie in $(-1,1)$, and strictly interlace by the classical theory of Jacobi
polynomials; see~\cite[Section~1]{DriverJordaanMbuyi2008}.  The affine change
$x=1+2t$ therefore puts the zeros of $V_r$ in $(-1,0)$ and preserves strict
interlacing.  Reciprocal reversal preserves strict interlacing on the negative
axis, proving the same assertions for~\eqref{eq:twoPatternDescents}.
\end{proof}

\section{Further questions}\label{sec:questions}

The proof of Theorem~\ref{thm:mainChow} works at fixed alphabet size and
arbitrary word length.  The noncrossing Chow family lies on the diagonal
$(r,m)=(n,n+1)$, which leads to several natural questions.

\begin{problem}\label{prob:consecutiveChowInterlacing}
Do $H_{\NC_{n+1}}(t)$ and $H_{\NC_{n+2}}(t)$ interlace for every $n$?
\end{problem}

\begin{problem}\label{prob:diagonalRecurrence}
Find a recurrence that combines alphabet extension and word extension along
$(r,m)=(n,n+1)$.  In particular, is there a useful scalar differential
recurrence for $H_{\NC_{n+1}}(t)$?
\end{problem}

The last-letter recurrence does not directly compare consecutive values of
$n$.  A transition that preserves a larger interlacing vector could strengthen
Theorem~\ref{thm:mainChow} to consecutive interlacing as in
Problem~\ref{prob:consecutiveChowInterlacing}.

The peak--tie recurrence invites a different refinement.

\begin{problem}\label{prob:peakTieDescentRefinement}
Refine Theorem~\ref{thm:parkingPeakTieRecurrence} by descents while retaining
a finite local state.
\end{problem}

\subsection*{Software}

The open-source package \texttt{polytool}~\cite{AlexanderssonPolytool}
implements recurrence searches and exact tests for real-rootedness and
interlacing.  The Lean library
\texttt{RealRooted}~\cite{AlexanderssonRealRooted} contains formalized results
on real-rootedness and interlacing.  In particular, it contains formal proofs
for the Narayana rows \oeis{A001263} and for the mutually reciprocal rows
\oeis{A033282} and \oeis{A126216}.  Neither package is needed for the proofs
above.

\subsection*{Acknowledgements}

We thank Petter Br\"and\'en for pointing out that the unrefined Smirnov
real-rootedness follows from the more general Toeplitz-matrix results
in~\cite{BrandenVecchi2025TN}.  We also thank Petter Br\"and\'en for first
drawing our attention to the connection between the last-letter recurrence
and Leander's work.  We thank Katharina Jochemko for later raising the same
connection and for comments that led us to clarify the interlacing terminology.
We used ChatGPT and Codex (OpenAI) for exploratory assistance and
editorial revision, and Claude (Anthropic) for an additional critical reading
of the manuscript.  All mathematical claims and proofs remain our
responsibility.

\clearpage
\appendix

\section{The boundary Euler inverse}
\label{app:boundaryEulerInverse}

\begin{lemma}[Boundary Euler inverse]\label{lem:boundaryEulerInverse}
Let $m\geq1$, $c>0$, and $\gamma>0$.  Define
\[
\defin{E_\gamma(x)}\coloneqq
{}_3F_2\!\left(
\begin{matrix}-m,c+m,\gamma\\c,\gamma+1\end{matrix};x\right).
\]
Here $(a)_0=1$ and $(a)_k=a(a+1)\dotsm(a+k-1)$ for $k\geq1$.  This
hypergeometric series terminates at degree $m$.  If $\gamma>c+m-1$, then for
every real polynomial $p$ of degree less than $m$,
\begin{equation}\label{eq:boundarySignedMoment}
\int_0^1x^{c-1}E_\gamma(x)p(x)\,dx
=-E_\gamma(1)\int_1^\infty p(u)u^{c-\gamma-1}\,du.
\end{equation}
Under the same inequality,
\begin{equation}\label{eq:boundaryEulerEndpoint}
E_\gamma(1)
=(-1)^m\frac{m!(\gamma-c-m+1)_m}{(c)_m(\gamma+1)_m},
\qquad
\operatorname{sgn}E_\gamma(1)=(-1)^m,
\end{equation}
and $E_\gamma$ has $m$ simple zeros, all in $(0,1)$.
\end{lemma}

\begin{proof}
Put \(\theta=x\,d/dx\) and
\(\defin{\Phi(x)}\coloneqq{}_2F_1(-m,c+m;c;x)\).  Since
$(\gamma)_k/(\gamma+1)_k=\gamma/(\gamma+k)$, termwise integration gives
\[
E_\gamma(x)=\gamma\int_0^1u^{\gamma-1}\Phi(ux)\,du.
\]
Thus $(\theta+\gamma)E_\gamma=\gamma\Phi$ and $E_\gamma(0)=1$, which
explains the term Euler inverse.
The Jacobi polynomial formula gives
$\Phi\doteq \JacP_m^{(c-1,0)}(1-2x)$, so $\Phi$ is orthogonal to every
polynomial of degree less than $m$ for the weight $x^{c-1}$ on $(0,1)$.
For $0\leq r<m$, orthogonality and one integration by parts give
\[
0=\int_0^1x^{c+r-1}\Phi(x)\,dx
=\frac1\gamma\left(
E_\gamma(1)+(\gamma-c-r)
\int_0^1x^{c+r-1}E_\gamma(x)\,dx\right).
\]
Since
$1/(\gamma-c-r)=\int_1^\infty u^{c-\gamma+r-1}\,du$, linearity gives
\eqref{eq:boundarySignedMoment}.  The Pfaff--Saalschütz summation gives
\eqref{eq:boundaryEulerEndpoint}, whose sign follows from
$\gamma-c-m+1>0$.

Suppose that $E_\gamma$ has only $s<m$ sign-changing zeros
$r_1,\dotsc,r_s$ in $(0,1)$.  Since $E_\gamma(0)=1$, its endpoint sign
implies $s\equiv m\pmod2$.  The test polynomial
\[
p_0(x)=(-1)^s\prod_{i=1}^s(x-r_i)
\]
makes the left side of~\eqref{eq:boundarySignedMoment} positive, whereas
$E_\gamma(1)p_0(u)>0$ for $u>1$, making the right side negative.  This is a
contradiction.  Hence $s=m$.  Since $E_\gamma$ has degree $m$, all its zeros
are simple and lie in $(0,1)$.
\end{proof}

\section{A separate proof of individual real-rootedness}
\label{app:independentToricRealRootedness}

We give a second proof of
Corollary~\ref{cor:toricContributionRealRooted} that does not use the
common-interlacer theorem.

\begin{proof}[Second proof of
Corollary~\ref{cor:toricContributionRealRooted}]
The assertion is immediate when \(n=1\), so assume \(m\geq1\) and use the
polynomials \(R_d\) from~\eqref{eq:RdDefinition}.

Let \(\boxtimes_m\) denote the degree-\(m\) finite multiplicative
convolution, also called Schur--Szegő composition.  If \(0\leq d<m\), the
hypergeometric convolution formula, Euler's transformation, and the Jacobi
polynomial formula give, up to a nonzero real scalar,
\begin{equation}\label{eq:RdSchurSzego}
R_d=\kappa_{m,d}
\left((1-x)^{d+1}
\JacP_{m-d-1}^{(\epsilon+d+1,d+1)}(1-2x)\right)
\boxtimes_m
\left((1-x)^{m-d}
\JacP_d^{(c-1,m-d)}(1-2x)\right),
\qquad \kappa_{m,d}\neq0;
\end{equation}
see~\cite[Theorem~3.1]{MartinezFinkelshteinMoralesPerales2024}.  All four
Jacobi parameters are greater than \(-1\), so both factors have only
nonnegative real zeros.  The Malo--Schur--Szegő closure theorem
\cite[Proposition~2.7]{MartinezFinkelshteinMoralesPerales2024} therefore
shows that \(R_d\) is real-rooted.

It remains to treat $d=m$.  Since
$m+1+\epsilon=c+m+\frac12$, the hypergeometric expressions
in~\eqref{eq:RdDefinition} and Lemma~\ref{lem:boundaryEulerInverse} give
$R_m=E_{c+m+1/2}$.  That lemma shows that $R_m$ has $m$ simple zeros in
$(0,1)$.  Thus every $R_d$ is real-rooted, and the observation following
\eqref{eq:RdDefinition} completes the proof.
\end{proof}

\bibliographystyle{alpha}
\bibliography{bibliography}

\end{document}